\documentclass[11pt,draft,amsmath,amscd,amsbsy,amssymb]{amsart}
 \relax
\chardef\bslash=`\\ 

\makeatletter \def\verbatim{\interlinepenalty\@M \@verbatim
\leftskip\@totalleftmargin\advance\leftskip2pc
\frenchspacing\@vobeyspaces \@xverbatim} \makeatother \hfuzz1pc
\makeatletter \def\dgt@k{\dg@DX=-3 \dg@DY=2 \dg@SIZE=3}
\makeatother
\makeatletter \def\dgt@kk{\dg@DX=3 \dg@DY=-1 \dg@SIZE=3}
\theoremstyle{plain} \newtheorem{thm}{Theorem}[section]

\newtheorem{lemma}[thm]{Lemma}

\theoremstyle{definition} \newtheorem{rem}[thm]{Remark}
 \newtheorem{que}[thm]{Question}

\newcommand{\mpcc}{\operatorname{\mathrm{mpcc}}}

\newcommand{\pr}{\mathrm{pr}}
\usepackage[all]{xy}
\begin{document}

\title[Hyperspaces of max-plus convex subsets]
{Hyperspaces of max-plus convex subsets of powers of the real line}
\author[L. Bazylevych]{Lidia Bazylevych}
\address{National University ``Lviv Polytechnica'', 12 Bandery Str.,
79013 Lviv, Ukraine} \email{izar@litech.lviv.ua}

\author[D. Repov\v s]{Du\v san Repov\v s}
\address{Faculty of Education and
Faculty of Mathematics and Physics, 
University of Ljubljana, 
P.O. Box 2964,  Ljubljana, Slovenia 1001}
\email{dusan.repovs@guest.arnes.si}

\author[M. Zarichnyi]{Mykhailo Zarichnyi}
\address{Department of Mechanics and Mathematics,
Lviv National University, Universytetska Str. 1, 79000 Lviv,
Ukraine}
\address{Institute of Mathematics,
University of Rzesz\'ow, Rejtana 16 A, 35-310 Rze\-sz\'ow, Poland}
\email{topology@franko.lviv.ua}
\thanks{}
\date{\today}


\begin{abstract} The notion of max-plus convex subset of Euclidean space can be naturally extended to other linear spaces.
The aim of this paper is to describe the topology of  hyperspaces of max-plus convex subsets of  Tychonov powers $\mathbb R^\tau$ of the real line. We show that the corresponding spaces are AR's if and only if $\tau\le\omega_1$.
\end{abstract}
\maketitle
\section{Introduction}
Max-plus convex sets were introduced in \cite{Zim}. Max-plus convex cones also appeared
in  idempotent analysis, after the observation
by Maslov that solutions of the
Hamilton-Jacobi equation associated
with a deterministic optimal control problem satisfy a ``max-plus'' superposition
principle and therefore
belong to structures similar to convex cones which are
called semimodules or idempotent linear spaces \cite{LMS}. In the last decade the interest in  max-plus convex sets increased 
due to
the development of the so-called ``idempotent mathematics'', which is a part of mathematics where usual arithmetic operations are replaced by idempotent operations.
Our paper is devoted to  hyperspaces of max-plus convex subsets in  products of real lines. The results of the first-named author cover the case of $\mathbb R^n$, $n\ge2$.

Topology of hyperspaces of compact and closed convex sets has been
investigated by several
authors. The classical result of Nadler,
Quinn and Stavrakas \cite{NQS} asserts that the hyperspace of
convex compact subsets of $\mathbb R^n$, $n\ge2$, is a
contractible $Q$-manifold homeomorphic to $Q\setminus\{*\}$ 
(recall that a $Q$-{\it manifold} is a
manifold modeled on the Hilbert cube $Q=[0,1]^\omega$). Their
result has
found many applications in convex geometry. In
particular, it enabled the proof that the hyperspace of all
compact strictly convex bodies is homeomorphic to the separable
Hilbert space $\ell^2$ (see \cite{B}). 
Hyperspaces of compact convex subsets of 
Tychonov cubes were investigated in \cite{ZI}.

Let $\mathbb R_{\max}=\mathbb R\cup\{-\infty\}$ and let $\tau$ be a cardinal number. Given $x,y\in\mathbb R^\tau$ and
$\lambda\in\mathbb R$, we denote by $x\oplus y$ the coordinatewise
maximum of $x$ and $y$ and by $\lambda\odot x$ the vector obtained
from $x$ by adding $\lambda$ to every its coordinate.   A subset
$A$ in $\mathbb R^\tau$ is said to be {\em tropically convex} (or {\em max-plus convex}) if
$\alpha\odot a\oplus\beta\odot b\in A$ for all $a,b\in A$ and
$\alpha,\beta\in\mathbb R_{\max}$ with $\alpha\oplus\beta=0$.

We denote  the hyperspace of all nonempty
max-plus convex compact subsets in $\mathbb R^\tau$ by $\mpcc(\mathbb R^\tau)$. Note that every
max-plus convex compact subset in $\mathbb R^\tau$ is a
subsemilattice of  $\mathbb R^\tau$ with respect to the operation
$\oplus$. In particular, $\max A\in A$, for any max-plus convex
compact subset $A$ in $\mathbb R^\tau$.

Tychonov powers $\mathbb R^\tau$, for $\tau>\omega$, are the
main geometric objects of the theory of noncompact nonmetrizable absolute extensors.
The main result of our paper is that the hyperspace of 
max-plus convex subsets in the spaces $\mathbb R^\tau$  is homeomorphic to $\mathbb R^\tau$ if $\tau\in\{\omega,\omega_1\}$.

\section{Preliminaries}

The set $\mathbb R\cup\{-\infty\}$ is endowed with the metric
$\varrho$, $\varrho(x,y)=|e^x-e^y|$ (conventions: $e^{-\infty}=0$
and
$\ln0=-\infty$).
 We denote the set of all nonempty compact subsets of a
metric space $(X,d)$ by $\exp X$. The base of the {\it Vietoris topology} on $\exp X$ consists of
the sets of the form $$\langle U_1,\dots,U_n\rangle=\{A\in\exp X\mid A\subset \cup_{i=1}^nU_i,\ A\cap U_i\neq\emptyset\text{, for all }i=1,\dots,n\},$$
where $U_1,\dots,U_n$ run over the topology of $X$.

If $X$ is a metric space, then one can endow $\exp X$ with the {\it Hausdorff metric}
$d_H$:
$$d_H(A,B)=\inf\{\varepsilon>0\mid A\subset O_\varepsilon(B),\ B\subset O_\varepsilon(A)\}$$
(hereafter, $O_r(C)$ will denote the $r$-neighborhood of $C\in\exp X$). It is well-known that equivalent metrics on $X$ generate equivalent Hausdorff metrics on $\exp X$.

 By ANR (resp., AR) we shall 
 denote the class of {\it absolute neighborhood retracts}
for the class of metrizable spaces, i.e. the class of metrizable spaces $X$ satisfying the following property: for every embedding $i\colon X\to Y$ into a metrizable space $Y$ there exists a retraction of a neighborhood of $i(X)$ in $Y$ (resp., a retraction of $Y$) onto $i(X)$.

We say that a metric space $X$ satisfies the {\em strong discrete
approximation property} (SDAP) if for every continuous function
$\varepsilon\colon X\to(0,\infty)$ and every map $f\colon \coprod_{n=1}^\infty I^n\to X$  there exists a map
$g\colon \coprod_{n=1}^\infty I^n\to X$ such that $d(f(x),g(x))<\varepsilon(x)$, $x\in \coprod_{n=1}^\infty I^n$, and the family $\{g(I^n)\mid n\in\mathbb N\}$ is discrete ($d$ denotes the metric on $X$).
The following is a characterization theorem for $\ell^2$-manifolds.

\begin{thm}[Toru\'nczyk \cite{T}] A complete separable nowhere locally compact ANR $X$ is an $\ell^2$-manifold
if  and only if $X$ satisfies the SDAP.
\end{thm}

Recall that a map $f\colon X\to Y$ is called  {\em soft}
\cite{Shch}, if for every commutative diagram
\begin{equation}\label{d:1} \xymatrix{A\ar@{^{(}->}[d]\ar[r]^\varphi&X\ar[d]^f\\
Z\ar[r]_\psi& Y}\end{equation} such that $A$ is a closed subset of a paracompact space $Z$, there exists a map $\Phi\colon Z\to X$ such that
$f\Phi=\psi$ and $\Phi|A=\varphi$.

A {\em trivial $\ell^2$-bundle} is a map $f\colon X\to Y$
which is homeomorphic to the projection $Y\times M\to Y$ onto the
first factor, where $M$ is  $\ell^2$.
A map $f\colon X\to Y$ of  metric spaces is said to satisfy
the {\em fiberwise discrete approximation property} if for every map $g\colon \coprod_{n=1}^\infty I^n\to X$ and every continuous function
$\varepsilon\colon X\to(0,\infty)$ there is a map $h\colon \coprod_{n=1}^\infty I^n\to X$   such that $d(f(x),g(x))<\varepsilon(x)$, $x\in \coprod_{n=1}^\infty I^n$, and
\begin{enumerate}
\item $fg=fh$; and
\item the family $\{h(I^n)\mid i\in\mathbb N\}$ is discrete.
\end{enumerate}
The following result was cited in \cite{Chi} and was attributed to Toru\'nczyk and West.

\begin{thm}[Toru\'nczyk-West characterization theorem for $\mathbb R^\omega$-manifold
bundles]\label{t:bundle} A map $f\colon X\to Y$ of complete metric
ANR-spaces is a trivial $\mathbb R^\omega$ if $f$ is soft and $f$
satisfies the fiberwise discrete approximation property (FDAP).
\end{thm}

The following notion was introduced in \cite{Ho}: a {\em
$c$-structure} on a topological space $X$ is an assignment, to
every nonempty finite subset $A$ of $X$, of a contractible subspace
$F(A)$ of $X$, such that $F(A)\subset F(A')$, whenever $A\subset
A'$. A pair $(X,F)$, where $F$ is a $c$-structure on $X$, is called
a {\em $c$-space}. A subset $E$ of $X$ is called an {\em $F$-set}
if $F(A)\subset E$ for any finite $A\subset E$. A metric space
$(X,d)$ is said to be a {\em metric $l.c.$-space} if all the open
balls are $F$-sets and all open $r$-neighborhoods of $F$-sets are
also $F$-sets.

The following is generalization of the Michael Selection Theorem 
for generalized convexity structures,
see \cite{Ho1} for the proof. Recall that a multivalued map
$F\colon X\to Y$ of topological spaces is called {\em lower
semicontinuous} if, for any open subset $U$ of $Y$, the  set
$\{x\in X\mid F(x)\cap U\neq\emptyset\}$ is open in $X$. A {\em
selection} of a multivalued map $F\colon X\to Y$ is a
(single-valued) map $f\colon X\to Y$ such that $f(x)\in F(x)$, for
every $x\in X$.
The following  was proved in \cite{Ho1} (see the second corollary
of Theorem 2).

\begin{thm}\label{t:ho1} Let $(X,d,F)$ be a metric $l.c.$-space. Then
$X$ is an AR.
\end{thm}

\begin{thm}\label{t:ho} Let $(X,d,F)$ be a complete metric $l.c.$-space. Then
any lower semicontinuous multivalued map $T\colon Y\to X$ of a
paracompact space $Y$ whose values are nonempty closed $F$-sets,
has a continuous selection.
\end{thm}

\section{Two lemmas}

Recall that the countable power $\mathbb R^\omega$  of the real line $\mathbb R$ is homeomorphic to the pseudo-interior $s$ of the Hilbert cube $Q$ as well as, by the Anderson-Kadec theorem, to the separable Hilbert space $\ell^2$. We shall consider the following metric $\varrho$ on $\mathbb R$: $$\varrho(x,y)=\min\{|x-y|,1\}.$$ We shall define a metric $d$ on the countable power $\mathbb R^\omega$ by the formula $$d((x_i)_{i=1}^\infty,(y_i)_{i=1}^\infty)=\max_{i\in\mathbb N}\frac{\varrho(x_i,y_i)}{2^i}.$$ Note that $d$ is a complete metric generating the Tychonov topology on $\mathbb R^\omega$.

\begin{lemma}\label{l:1} The space $\mpcc(\mathbb R^\omega)$ is an absolute retract.
\end{lemma}
\begin{proof}  Define a $c$-structure on $\mpcc(\mathbb R^\omega)$ as follows:
given any $A_1,\dots, A_n\in \mpcc(\mathbb
R^\omega)$, let 
$$F(\{A_1,\dots, A_n\})=\left\{\bigoplus_{i=1}^n\alpha_i\odot A_i\mid \alpha_1,\dots,\alpha_n\in[-\infty,0],\ \oplus_{i=1}^n\alpha_i=0\right\}.$$

We are going to show that every set of the form $F(\{A_1,\dots, A_n\})$ is contractible. Let $A=\oplus_{i=1}^nA_i$. Then $A\in F(\{A_1,\dots, A_n\})$. Define a map $$H\colon F(\{A_1,\dots, A_n\})\times [0,1]\to F(\{A_1,\dots, A_n\})$$ by the formula: $$H(C,t)=C\oplus(\ln t)\odot A.$$ Note that $H$ is well-defined,  $H(C,0)=C\oplus \{-\infty\}=C$ and $H(C,1)=C\oplus 0\odot A=A$, for every $C\in F(\{A_1,\dots, A_n\})$. Thus, $H$ contracts the set $F(\{A_1,\dots, A_n\})$ to $A$.

Now let us prove that every neighborhood of a point in $\mpcc(\mathbb R^\omega)$ is an $F$-set. Let $A\in \mpcc(\mathbb R^\omega)$, $r>0$, and $d_H(A,B),d_H(A,B')<r$. Given $a\in A$, find $b\in B$ and $b'\in B'$ such that $d(a,b)<r$ and $d(a,b')<r$. Without loss of generality, we may assume that $a=0$. There exist $i,j\in \mathbb N$ such that $$d(a,b)=\frac{\min\{|b_i|,1\}}{2^i},\ d(a,b')=\frac{\min\{|b'_j|,1\}}{2^j}.$$
Given $t\in [-\infty,0]$, find $k\in\mathbb N$ such that $$d(a,b\oplus t\odot b')= \frac{\min\{|\max\{b_k, b'_k+t\}|,1\}}{2^k}.$$

Without loss of generality, we may assume that $r<1$. The rest of the proof splits in two cases.

{\it Case 1.}
$b_k\ge b'_k+t$. Then $$d(a,b\oplus t\odot b')= \frac{|b_k |}{2^k}\le \frac{|b_i |}{2^i}<r.$$

{\it Case 2.}
$b_k\le b'_k+t$. Then also $ b'_k+t\le b'_k$ and
$$d(a,b\oplus t\odot b')= \frac{|b_k +t|}{2^k}\le \max\left\{\frac{|b_k |}{2^k}, \frac{|b_k |}{2^k}\right\}\le \max\left\{\frac{|b_i |}{2^i}, \frac{|b_j |}{2^j}\right\}
<r.$$

In both cases, for every $a\in A$ there is a point $c\in  B\oplus t\odot B'$ such that $d(a,c)<r$. Similarly, for any $c\in  B\oplus t\cdot B'$ one can find $a\in A$ such that $d(a,c)<r$. This shows that $d_H(A, B\oplus t\odot B')<r$ for every $B,B'\in \mpcc(\mathbb R^\omega)$ such that $d_H(A, B)<r$, $d_H(A, B')<r$.

We can demonstrate by induction that
$$d_H(\bigoplus_{i=1}^n\alpha_i\odot A_i,A)<r, \ \hbox{\rm whenever} \ \ \alpha_1,\dots,\alpha_n\in[-\infty,0],$$ 
$$\oplus_{i=1}^n\alpha_i=0, \ \hbox{\rm  and} \ \  d_H(A_i,A)<r, \ \hbox{\rm for every}
\ i=1,\dots,n.$$
This shows that every $r$-neighborhood of a point in the space $\mpcc(\mathbb R^\omega)$ is an $F$-set.

By using a
similar argument we can prove that every neighborhood of an $F$-set is again an $F$-set. It follows from the results of \cite{Ho1}
that the space $\mpcc(\mathbb R^\omega)$ is an AR-space (see Theorem \ref{t:ho1} above). \end{proof}

Let $A,B$ be nonempty sets with $A\subset B$.  
Observe
that the projection $p=p^B_A\colon \mathbb
R^B\to \mathbb R^A$  onto the first factor induces the map $$\mpcc(p)\colon \mpcc(\mathbb
R^B)\to \mpcc(\mathbb R^A)$$ as follows: $$\mpcc(p)(A)=p(A),\ A\in \mpcc(\mathbb
R^B).$$ It is easy to verify that this map is well-defined. 

We may regard the construction $\mpcc$ as a covariant functor acting on the category whose objects are the powers of $\mathbb R$ and the morphisms are the projections. 

 \begin{lemma} Let $p\colon \mathbb
R^\omega\times\mathbb
R^\omega\to \mathbb R^\omega$ be the projection onto the first factor. Then  the  map $\mpcc(p)$ is soft.
\end{lemma}

\begin{proof} Consider a commutative diagram
\begin{equation}\label{d:9} \xymatrix{A\ar@{^{(}->}[d]\ar[r]^{\varphi\ \ \ \ }&{\mpcc(\mathbb
R^\omega\times\mathbb
R^\omega)}\ar[d]^{\mpcc(p)}\\
Z\ar[r]_{\psi\ \ \ }&  \mpcc(\mathbb R^\omega),}\end{equation}
where $A$ is a closed subset of a paracompact space $Z$.

For every $C\in \mpcc(\mathbb R^\omega)$, the preimage $$\mpcc(p)^{-1}(c)\subset \mpcc(\mathbb
R^\omega\times\mathbb
R^\omega)$$ is convex with respect to the $c$-structure $F$ in the space $\mpcc(\mathbb
R^\omega\times\mathbb
R^\omega)$ defined as follows:
given any $A_1,\dots, A_n\in \mpcc(\mathbb
R^\omega\times\mathbb
R^\omega)$, let
$$F(\{A_1,\dots, A_n\})=\left\{\bigoplus_{i=1}^n\alpha_i\odot A_i\mid \alpha_1,\dots,\alpha_n\in[-\infty,0],\ \oplus_{i=1}^n\alpha_i=0\right\}.$$ Note that this is an $F$-structure with respect to the Hausdorff metric $d'_H$ on the space $\mpcc(\mathbb R^\omega\times \mathbb R^\omega)$ generated by the metric $d'$ on the space $\mathbb R^\omega\times \mathbb R^\omega$ defined by the formula: $$d'((x_1,y_1),(x_2,y_2))=\max\{d(x_1,x_2),d(y_1,y_2)\}.$$
This can be established by repeating the corresponding arguments from the proof of Lemma \ref{l:1}.

Define a multivalued map $\Phi\colon Z\to \mpcc(\mathbb
R^\omega\times\mathbb
R^\omega)$ as follows: $$\Phi(z)=\begin{cases} \mpcc(p)^{-1}(\psi(z)), & \text{ if }z\in Z\setminus A,\\
\{\varphi(z)\}, \text{ if }z\in A.
\end{cases}$$
Clearly, the images of $\Phi$ are $F$-sets. Since the set $A$ is closed, we see that the map $\Phi$ is lower semicontinuous. It follows from Selection Theorem \ref{t:ho} that this map admits a continuous selection, $g$. Clearly, $g|A=\varphi$ and $g\mpcc(p)=\psi$. This proves the softness of $\mpcc(p)$.
\end{proof}

\section{The main result}

\begin{thm}\label{mpt19}
The hyperspace $\mpcc(\mathbb
R^\omega)$ of compact max-plus  convex subsets in the space $\mathbb
R^\omega$  is homeomorphic to $\mathbb R^\omega$.
\end{thm}

\begin{proof} Since $\mathbb
R^\omega$ is homeomorphic to $(\mathbb
R^\omega)^\omega$, one can represent the latter space as the limit of the inverse sequence $$\mathbb
R^\omega \leftarrow \mathbb
R^\omega \times \mathbb
R^\omega \leftarrow \mathbb
R^\omega \times \mathbb
R^\omega \times \mathbb
R^\omega \leftarrow\dots,$$ where every arrow denotes the projection onto the first factor. Applying the functor $\mpcc$ to this sequence we obtain \begin{equation}\label{e:1}
\mpcc(\mathbb
R^\omega) \leftarrow \mpcc(\mathbb
R^\omega \times \mathbb
R^\omega) \leftarrow \mpcc(\mathbb
R^\omega \times \mathbb
R^\omega \times \mathbb
R^\omega) \leftarrow\dots\end{equation} The bonding maps of the latter sequence have the following property: for every such a map there exists a countable family of selections such that the family of images of these selections is discrete. Indeed, let $C=\{c_i\mid i\in\omega\}$ be a closed countable subset of $\mathbb R^\omega$. For every $i\in\omega$, denote by $s_i$ the selection of the map $$\mpcc((\mathbb
R^\omega \times \mathbb
R^\omega \times\dots\times  \mathbb
R^\omega)\times\mathbb
R^\omega)\to\mpcc(\mathbb
R^\omega \times \mathbb
R^\omega \times\dots\times  \mathbb
R^\omega)$$
defined as follows: $s_i(A)=A\times\{c_i\}$.

We are going to show that the the limit projection of the inverse limit of (\ref{e:1}) onto $\mpcc(\mathbb R^\omega)$ satisfies the FDAP. Let $f\colon \sqcup_{i\in\mathbb N}Q_i\to \mpcc((\mathbb R^\omega)^\omega)$ be a map and let $\varepsilon \colon \mpcc((\mathbb R^\omega)^\omega)\to(0,\infty)$ be a function. For every $n\in\omega$, let $$Y_n=\left\{y\in Y\mid \varepsilon (f(y))\ge \frac{1}{2^n}\right\}.$$ Note that $$Y_0\subset \mathrm{Int}(Y_1)\subset Y_1\subset \mathrm{Int}(Y_2)\subset Y_2\dots.$$

Define, for every $l=0,2,4,\dots$, a map $g_l\colon Y_{l-1}\cup Y_l\cup Y_{l+1}\to \mpcc((\mathbb R^\omega)^\omega)$ by the formula: $$g_l(y)=\mpcc(\pr_{l+1})(f(y))\times\{c_i\}\times\{c_i\}\times\dots,$$
whenever $y\in Q_i$.
Now, for every $l=1,3,5,\dots$, let $\varphi_l\colon Y_{l-1}\cup Y_l\cup Y_{l+1}\to [0,1]$ be a function such that $\varphi_l|Y_{l-1}\equiv 0$, $\varphi_l|Y_{l+1}\equiv 0$.

Define a map $g\colon \sqcup_{i\in\mathbb N}Q_i\to \mpcc((\mathbb R^\omega)^\omega)$ by the following condition. Let $y\in Y_l$, where $l=0,2,4,\dots$. Then define $g(y)=g_l(y)$. If $y\in Y_l\cap Q_i$, where $l=1,3,5,\dots$, then define
$$g(y)=\{(a_1,\dots,a_l, \varphi_l(y)a_{l+1}+(1-\varphi_l(y))c_i,c_i,c_i,\dots)\mid (a_k)_{k=1}^\infty\in f(y)\}.$$
It is easy to see that the map $g$ is well-defined, $\mpcc(\pr_1)f=\mpcc(\pr_1)g$, and that $d(f(x),g(x))<\varepsilon(x)$, for every $x\in \sqcup_{i\in\mathbb N}Q_i$.

We are going to prove that the map $g$ is a closed embedding. Suppose the contrary. Then there exists a sequence $(y_{k_i})_{i=1}^\infty$, where $y_{k_i}\in Q_{k_i}$ for every $i$ (here we assume that $k_1<k_2<k_3<\dots$), such that $\lim_{i\to\infty}g(y_{k_i})=A$, for some  $A\in\mpcc((\mathbb R^\omega)^\omega)$. Without loss of generality, one may assume that $k_i=i$, for all $i$.

Since $\varepsilon(A)>0$, one may assume that $\varepsilon(g(y_i))>2^{-n}$, for some $n<\omega$.
Denote by $\pi_k\colon (\mathbb R^\omega)^\omega\to \mathbb R^\omega$ the projection onto the $k$th factor. Then  from the construction of the map $g$ it follows that $\mpcc(\pi_{n+1})(g(y_i))=\{c_i\}$. Since the set $C=\{c_i\mid i\in\omega\}$ is closed in $\mathbb R^\omega$, we obtain a contradiction.

It now follows from Theorem \ref{t:bundle} that the limit projection of the inverse limit of the inverse sequence (\ref{e:1}) onto $\mpcc(\mathbb R^\omega)$ is a trivial $\ell^2$-bundle.
Since the space $\mpcc(\mathbb R^\omega)$ is an absolute retract, we conclude that $$\mpcc(\mathbb R^\omega)\simeq \mpcc((\mathbb R^\omega)^\omega)\simeq\mpcc(\mathbb R^\omega) \times \ell^2\simeq  \ell^2,$$ which proves the theorem.
\end{proof}

\begin{rem}\label{r:1} As a by-product of the proof we see that the map $\mpcc(p_1)\colon \mpcc(\mathbb R^\omega\times\mathbb R^\omega)\to \mpcc(\mathbb R^\omega)$ is a trivial $\ell^2$-bundle (here $p_1$ denotes the projection onto the first factor).
\end{rem}

The following result is an analog of a theorem of the first-named author \cite{Ba1},
proved for the open sets in $\mathbb R^n$, $n\ge2$.

\begin{thm}\label{mpt10} Let $X$ be an open subset in the space $\mathbb R^\omega$. Then the hyperspace of max-plus  convex subsets in $X$  is homeomorphic to $X$.
\end{thm}

\begin{proof} The set $X$ is a  $\mathbb R^\omega$-manifold being an open subset of $\mpcc(\mathbb R^\omega)$. We identify the set  $X$ with the set of all singletons in
$X$. The map  $\max\colon\mpcc(X)\to X$ is therefore a retraction. Denote the homotopy $H\colon \mpcc(X)\times [0,1]\to
\mpcc(X)$ by the formula $$H(A,t)=\{a\oplus \ln t\max A\mid a\in A\},\
A\in\mpcc(X),\ t\in[0,1]$$ (convention: $\ln0=-\infty$).

Therefore, the space $X$ is a deformation retract of the space $\mpcc(X)$,
whence we conclude that the spaces $X$ and $\mpcc(X)$ are homotopically equivalent. The classification theorem for  $\mathbb
R^\omega$-manifolds implies that the spaces  $X$ and $\mpcc(X)$
are homeomorphic.
\end{proof}

\begin{thm} The hyperspace $\mpcc(\mathbb R^{\omega_1})$ is homeomorphic to $\mathbb R^{\omega_1}$.
\end{thm}

\begin{proof} We represent $\mathbb R^{\omega_1}$ as the limit of the  inverse system $\mathcal S=\{(\mathbb R^{\omega})^\alpha, p_{\alpha\beta}; \omega_1\}$, where, for $\alpha>\beta$, the map $p_{\alpha\beta}\colon (\mathbb R^{\omega})^\alpha\to (\mathbb R^{\omega})^\beta$ is the projection map. Then, recall that every projection map $p_{\alpha\beta}$ induces the map $\mpcc(p_{\alpha\beta})\colon \mpcc((\mathbb R^{\omega})^\alpha)\to \mpcc((\mathbb R^{\omega})^\beta)$ and therefore we obtain the inverse system $$\mpcc(\mathcal S)=\{\mpcc((\mathbb R^{\omega})^\alpha), \mpcc(p_{\alpha\beta}); \omega_1\}.$$

Since by Remark \ref{r:1}, every bonding map $\mpcc(p_{\alpha\beta})$ is homeomorphic to the projection $p\colon \mathbb
R^\omega\times\mathbb
R^\omega\to \mathbb R^\omega$, we conclude that  $$\mpcc(\mathbb R^{\omega_1})=\mpcc(\varprojlim(\mathcal S))= \varprojlim(\mpcc(\mathcal S))\simeq \mathbb R^{\omega_1}$$
(the second equality is simply the continuity of the functor $\mpcc$; see \cite{Ba} for details.)
\end{proof}
In the sequel, we shall
speak of the theory of noncompact nonmetrizable absolute extensors in the sense of \cite{Chi}. They are defined as retracts of functionally open subspaces of powers of the real line. Recall that a set $U$ in a topological space $X$ is called {\em functionally open} if $U=f^{-1}((0,1])$, for some continuous function $f\colon X\to[0,1]$. 

\begin{thm} Let  $M$ be a functionally  open subset of $\mathbb R^{\omega_1}$. Then $\mpcc(M)$ is homeomorphic to $M$.
\end{thm}

\begin{proof} Note that the set $\mpcc M$ is also functionally open. Indeed, let $f\colon  \mathbb R^{\omega_1}\to[0,1]$ be a continuous function such that $M=f^{-1}((0,1])$. Define the function $\tilde f\colon \mathbb R^{\omega_1}\to[0,1]$ by the formula: $\tilde f(A)=\inf A$. Then, clearly, $\tilde f^{-1}((0,1])=\mpcc M$.

There exists a countable subset $S\subset \omega_1$ and a function $g\colon \mathbb R^S\to[0,1]$ such that $f=g\pr_S$. Therefore, $M=U\times \mathbb R^{\omega_1\setminus S}$. Without loss of generality, one may conclude that $S=\omega\subset\omega_1$.
We conclude that $$M=\varprojlim\{U\times \mathbb R^{\alpha\setminus\omega}, p^{\beta\setminus\omega}_{\alpha\setminus\omega};  \omega<\alpha<\beta<\omega_1 \}$$ and therefore
$$\mpcc(M)=\varprojlim\{\mpcc(U\times \mathbb R^{\alpha\setminus\omega}), \mpcc(p^{\beta\setminus\omega}_{\alpha\setminus\omega});  \omega<\alpha<\beta<\omega_1 \}.$$

Since by Theorem \ref{mpt10}, the space $\mpcc(U)$ is homeomorphic to $U$ and every projection map in the latter inverse system is soft, we conclude that $$\mpcc(M)\simeq \mpcc(U)\times \mathbb R^{\omega_1}\simeq U \times \mathbb R^{\omega_1} \simeq M.$$
\end{proof}

\begin{thm} The hyperspace $\mpcc(\mathbb R^{\tau})$ is not an absolute retract, for any $\tau>\omega_1$.
\end{thm}

\begin{proof} First, note that it  suffices to consider the case $\tau=\omega_2$.
Mow, recall  that $\mpcc$ is a functor acting on the category whose objects are spaces $\mathbb R^\tau$ and the morphisms are the projections. Assuming that  $\mpcc(\mathbb R^{\omega_2})$ is an absolute retract we conclude, by Chigogidze's characterization theorem \cite{Chi1}, that $\mpcc(\mathbb R^{\omega_2})$ is homeomorphic to $\mathbb R^{\omega_2}$.

By   general results concerning the functors in the category of Tychonov spaces \cite{Chi,Shch}, we obtain that any homeomorphism of $\mathbb R^{\omega_2}$ and $\mpcc(\mathbb R^{\omega_2})$ implies the isomorphism of the square diagram $$\mathcal D={\xymatrix{{(\mathbb R^\omega)}^3 \ar[r]^{\pr_{12}} \ar[d]_{\pr_{13}}& {(\mathbb R^\omega)}^2 \ar[d]^{\pr_{1}}\\ {(\mathbb R^\omega)}^2 \ar[r]_{\pr_{1}}& {\mathbb R^\omega} }},$$ where $\pr_{ij}$, $\pr_k$ denote the projections onto the corresponding factors,  and $\mpcc(\mathcal D)$.

We are going to show that the diagram $\mpcc(\mathcal D)$ is not a pullback diagram. Let $$A=\{0\}\subset \mathbb R^\omega,\ B=C=\{0\}\times \{(x_i)\mid x_0\in[0,1],\ x_i=0\text{, if }i>0\}\subset {(\mathbb R^\omega)}^2.$$ Let also $$D_1=\{0\}\times \{((x_i),(y_i))\mid x_0=y_0\text{ and } x_i=y_i=0\text{, if }i>0\}\subset {(\mathbb R^\omega)}^3,$$ then $$\mpcc(\pr_{12})(D)=\mpcc(\pr_{12})(D_1)=B,\ \mpcc(\pr_{13})(D)=\mpcc(\pr_{13})(D_1)=C.$$
Thus $\mpcc(\mathcal D)$ is not a pullback diagram and this completes the proof.
\end{proof}

\section{Epilogue}
The following question is related to Theorem \ref{mpt10}.
\begin{que}
Let $U$ be an open subset of $\mathbb R^{\omega_1}$ which is an $\mathbb R^{\omega_1}$-manifold (see \cite{Chi1} for the background of the theory of $\mathbb R^{\omega_1}$-manifolds). Is
then
$\mpcc(U)$ homeomorphic to $U$?
\end{que}

The following notion was introduced in \cite{BH1} and investigated in \cite{BH} and \cite{BH2}.
A subset $B$ of $\mathbb R^n_+$ is said to be $\mathbb B$-{\it convex} if for all $x,y \in B$ and all $t \in [0, 1]$ one has $\max(tx, y) \in B$. For the hyperspace $\mathbb B\text{-cc}(\mathbb R^n)$, $n\ge2$, of compact $\mathbb B$-convex subsets of $\mathbb R^n_+$ one can prove analogues of the results in \cite{Ba}.

One can extend this notion over an arbitrary vector lattice. Let $\ell^2_+$ denote the positive cone of the separable Hilbert space $\ell^2$. We say that a  subset $B$ of $\ell^2_+$ is $\mathbb B$-{\it convex} if for all $x,y \in B$ and all $t \in [0, 1]$ one has $\max(tx, y) \in B$.
We conjecture that the hyperspace of compact $\mathbb B$-convex subsets in $\ell^2_+$ is homeomorphic to $\ell^2$.
Analogous question can be formulated for
the nonseparable case.
\begin{que} Let $\ell^2(A)_+$ denote the positive cone  in a nonseparable Hilbert space $\ell^2(A)$. Is the hyperspace $\mathbb B\text{-cc}(\ell^2(A)_+)$ homeomorphic to $\ell^2(A)$?
\end{que}

\section*{Acknowledgements}

This research was supported by the
Slovenian Research Agency grants P1-0292-0101, 
J1-2057-0101, and J1-4144-0101.
We thank the referee for  comments and suggestions.


\begin{thebibliography}{99}
\bibitem{B} 
L. E. Bazylevych, 
{\it On the hyperspace of strictly convex bodies},
Mat. Stud. {\bf 2}(1993), 83--86.

\bibitem{Ba} 
L. E. Bazylevych,
{\it Hyperspaces of max-plus  convex compact sets},
Mat. Zametki {\bf 84}:5 (2008), 658–-666. (in Russian)

\bibitem{Ba1} 
L. E. Bazylevych, 
{\it Hyperspaces of max-plus and max-min convex sets,}
Travaux Math\'ematiques {\bf 18}(2008), 103--110.

\bibitem{BH1}
W. Briec and C. D. Horvath,
{\it $\mathbb B$-convexity},
Optimization {\bf 53}:2 (2004), 103--127.

\bibitem{BH}
W. Briec and C. Horvath,
{\it Halfspaces and Hahn-Banach like properties in $\mathbb B$-convexity and max-plus convexity},	 	
Pacif. J. Optim. {\bf 4}:2 (2008),  293--317.

\bibitem{BH2}
W. Briec, C. D. Horvath and A. M. Rubinov,
{\it Separarion in $\mathbb B$-convexity},
Pacif. J. Optim. {\bf 1}:1 (2005), 13--30.

\bibitem{Chi} 
A. Chigogidze, 
{\it Trivial bundles and near-homeomorphisms},
Fund. Math. {\bf 132}:2 (1989), 89-–98.

\bibitem{Chi1}
A. Chigogidze,
{\it Inverse Spectra},
North-Holland Mathematical Library {\bf 53},
Amsterdam 1996.

\bibitem{Ho}
C. D. Horvath,
{\it Contractibility and generalized convexity},
J. Math. Anal. Appl.  {\bf 156}:2 (1991), 341--357.

\bibitem{Ho1} 
C. D. Horvath,
{\it Extension and selection theorems in
topological spaces with a generalized convexity structure},
Ann. Fac. Sci. Toulouse Math.  {\bf (6) 2}:2 (1993),  253--269.

\bibitem{LMS} 
G. L. Litvinov, V. P. Maslov and G. B. Shpiz,
{\it Idempotent functional analysis: An algebraic approach},
Mat. Zametki {\bf 69}:5 (2001),  758--797. (in Russian);
English transl.:  Math. Notes {\bf 69}:5-6 (2001), 696-–729.

\bibitem{NQS} 
S. B. Nadler, Jr.,  J. Quinn and N. M. Stavrakas, 
{\it Hyperspace of compact convex sets,}
Pacif. J. Math.  {\bf 83}:2 (1979),  441-–462.
    
\bibitem{Shch} E. V. Shchepin,
{\it Functors and uncountable powers of compacta},
Uspekhi Mat. Nauk {\bf 31}:3 (1981), 3--62.   (in Russian)
    
\bibitem{T}
H. Toru\'{n}czyk, 
{\it Characterizing Hilbert space topology,}
Fund. Math. {\bf 111} (1981), 247--262.

\bibitem{TW} 
H. Toru\'{n}czyk and J. West, 
{\it Fibrations and Bundles with Hilbert Cube Manifold Fibers,}
Memoirs Amer. Math. Soc. {\bf  406} (1989).

\bibitem{ZI} 
M. M. Zarichnyi and S. O. Ivanov,		
{\it  Hyperspaces of convex compact subsets of the Tikhonov cube},
Ukrain. Mat, \v{Z}ur. {\bf 53}:1 (2001), 698--701. (in Ukrainian);
English transl.: Ukrain. Math. J. {\bf 53}:5 (2001), 809-–813.

\bibitem{Zim}
K. Zimmermann,
{\it A general separation theorem in extremal algebras},
Ekonom.-Mat. Obzor. {\bf 13}:2 (1977),179-–201.

\end{thebibliography}
\end{document}